\documentclass[dvipdfmx,xcolor=dvipsnames, draft]{article}
\usepackage[utf8]{inputenc}

\usepackage[top=15truemm,bottom=15truemm,left=25truemm,right=25truemm]{geometry}

\usepackage{mathrsfs}

\usepackage{amsmath,amssymb}
\usepackage{mathtools}

\usepackage{color}

\usepackage{ulem}
\usepackage{colortbl}
\usepackage{arydshln}

\usepackage{multicol}

\numberwithin{equation}{section}

\usepackage{colortbl}

\def\diam{\mathrm{diam}}

\def\inter{\mathrm{Int}}
\def\Con{\mathrm{Cone}}

\usepackage{ascmac}

\usepackage{amsthm}
\theoremstyle{definition}
\newtheorem{theorem}{Theorem}[section]

\newtheorem{lemma}[theorem]{Lemma}
\newtheorem{corollary}[theorem]{Corollary}
\newtheorem{remark}[theorem]{Remark}
\newtheorem{definition}[theorem]{Definition}

\newtheorem*{theorem*}{Theorem}
\newtheorem*{lem*}{Lemma}
\newtheorem*{rem*}{Remark}
\newtheorem*{claim*}{Claim}

\usepackage{thmtools}

\usepackage{graphicx}				
\usepackage{here}						

\usepackage[abs]{overpic} 			

\title{Packing measure and dimension of the limit sets of IFSs of generalized complex continued fractions\footnote{2020 Mathematics Subject Classification: 28A78, 28A80}}
\date{\today}
\author{Kanji INUI\footnote{corresponding author}\\
Department of mechanical engineering and science, \\
Faculty of science and engineering, Doshisha University \\ 
1-3, Tataramiyakodani, Kyotanabe-shi, Kyoto, 610-0394, JAPAN \\
E-mail: kinui@mail.doshisha.ac.jp\\
 \\ 
 Hiroki SUMI \\Course of Mathematical Science, Department of Human Coexistence, \\
Graduate School of Human and Environmental Studies, Kyoto University\\
Yoshida-nihonmatsu-cho, Sakyo-ku, Kyoto, 606-8501, JAPAN\\
E-mail: sumi@math.h.kyoto-u.ac.jp\\
Homepage: http://www.math.h.kyoto-u.ac.jp/\textasciitilde sumi/index.html
}
\begin{document}

\maketitle

\begin{abstract}
We consider a family of conformal iterated function systems (for short, CIFSs) of generalized complex continued fractions. 
Note that in our previous paper we showed that the proper-dimensional Hausdorff measure of the limit set of each CIFS is zero and the packing measure of the limit set with respect to the Hausdorff dimension is positive. 
In this paper, we show that the packing dimension and the Hausdorff dimension of the limit set of each CIFS in the family are equal, and the proper-dimensional packing measure of the limit set is finite. 
%To prove the results, we consider three cases and define a `nice' subset of the index set of the CIFS in each case. 
%In addition, we estimate the cardinality of the `nice' subsets and prove properties of `nice' subset, to estimate the conformal measure. 
\footnote{Keywords: infinite conformal iterated function systems, fractal geometry, packing measures, packing dimension, generalized complex continued fractions.}
\end{abstract}
%\tableofcontents
%\clearpage
%
%
%
%
%
\setcounter{subsection}{0}
\section{Introduction and the main results}
%Fractal geometry has been developed in order to study the geometrical properties of fractals. 

One of the major studies in fractal geometry is the study of estimating the dimensions and measures of fractals. 
By estimating the dimensions and measures of fractals, it is possible to explain phenomena that appear in fractals which are different from the ones that appear in `usual figures' (see \cite{Mandel}). 
Therefore, fractal geometry have attracted attention not only in mathematics but also in many other fields. 

Mathematically speaking, iterated function systems are powerful tools to construct fractals (more precisely, limit sets) and there are many papers on the study of estimating the dimensions and measures of the limit sets constructed by iterated function systems. 
For example, Falconer's paper \cite{F1} and book \cite{F2} consider the general theory and examples of the limit sets constructed by conformal iterated function systems with finitely many mappings (for short, finite CIFSs) on intervals under the open set condition (for short, the OSC), and discuss the dimensions and measures of the limit sets.  
In addition, Mauldin's and Urba\'{n}ski's paper \cite{MU} presents the general theory of estimating the dimensions and measures of the limit sets constructed by finite CIFSs on general compact connected subsets of $\mathbb{R}^{d}$ under the OSC (and some conditions). 
By the formula and theorems in \cite{F1}, \cite{F2} and \cite{MU}, under the OSC (and some conditions), we deduce that the packing dimension of the limit set constructed by any finite CIFS equals the Hausdorff dimension of the limit set, the proper-dimensional Hausdorff measure of the limit set is positive and finite, and so is the proper-dimensional packing measure of the limit set. 
Here, it should be mentioned that the positivity and finiteness of the Hausdorff measure of the limit set of any finite CIFS with respect to the zero of the pressure function are equivalent to the OSC (see \cite{PRSS}), and there is an example of a self-similar set with zero Hausdorff measure but positive and finite packing measure without the OSC (see \cite{PSS}). 

In addition, Mauldin's and Urba\'{n}ski's paper \cite{MU} presents the general theory of estimating the dimensions and measures of the fractals (more precisely, the limit sets too) constructed by conformal iterated function systems with infinitely many mappings (for short, infinite CIFSs) under the OSC (and some conditions). 
Indeed, they generalize the above formula and theorems, and by the generalized theorems in their paper we may obtain the non-positivity of the proper-dimensional Hausdorff measure of the limit sets constructed by infinite CIFSs even if we assume the OSC.  
This theorem indicates that we may find a different phenomenon of infinite CIFSs from finite CIFSs under the OSC. 
Here, it should be mentioned that there are now many directions for the studies of the limit sets constructed by infinite CIFSs with the OSC (for example, see \cite{MU2}, \cite{IK}, \cite{MU3}, \cite{RU} and so on) or without the OSC (for example, see \cite{MiU}, \cite{MiU2}, and so on). 

Moreover, in Mauldin's and Urba\'{n}ski's paper \cite{MU}, they constructed an interesting example of an infinite CIFS (with the OSC) and the limit set which are related to the complex continued fractions. 
The precise construction of the example is the following. 
Let $X := \{ z \in \mathbb{C} \ | \ |z-1/2| \leq 1/2 \}$. 
We call $\hat{S} := \{ \hat{\phi}_{(m, n)}(z) \colon X \to X \ | \ (m, n) \in \mathbb{Z} \times \mathbb{N} \}$ the CIFS of complex continued fractions, where $\mathbb{Z}$ is the set of integers, $\mathbb{N}$ is the set of positive integers and
\begin{align*}
    \hat{\phi}_{(m, n)}(z) := \frac{1}{z + m + ni} \quad (z \in X). 
\end{align*}
Let $\hat{J}$ be the limit set of $\hat{S}$ (see Definition \ref{CIFSDEF}) and $\hat{h}$ be the Hausdorff dimension of $\hat{J}$. 
For each $s \geq 0$, we denote by $\mathcal{H}^{s}$ the $s$-dimensional Hausdorff measure and denote by $\mathcal{P}^{s}$ the $s$-dimensional packing measure. 
Regarding this example, Mauldin and Urba\'{n}ski showed the following theorem.
\begin{theorem}[D. Mauldin, M. Urba\'{n}ski (1996)] \label{knownmainresult}
Let $\hat{S}$ be the CIFS of complex continued fractions as above. 
Then, we have $\mathcal{H}^{\hat{h}}(\hat{J}) = 0$ and $0 < \mathcal{P}^{\hat{h}}(\hat{J}) < \infty$. 
\end{theorem}
Note that they obtained an example of infinite CIFSs for which the proper-dimensional Hausdorff measure of the limit set is zero and the proper-dimensional packing measure of the limit set is positive and finite.  
That is, they found an example of a phenomenon of infinite CIFSs with the OSC which cannot hold in any finite CIFSs with the OSC. 

In our previous papers \cite{IOS} and \cite{IS}, we considered a family of CIFSs of generalized complex continued fractions
%which is a generalization of the CIFS $\hat{S}$ of complex continued fractions, 
in order to present other examples of infinite CIFSs. 
In particular, we estimated the Hausdorff dimension of the limit set of each CIFS of the generalized complex continued fractions, and showed the non-positivity of the Hausdorff measure and the positivity of the packing measure of the limit set with respect to the Hausdorff dimension of the limit set. 
Note that the family of the CIFSs introduced in the papers \cite{IOS} and \cite{IS} has uncountably many elements. 
However, there are unsolved problems in \cite{IOS} and \cite{IS}: the estimation of the packing dimension and the proper-dimensional packing measure of the limit set.   
Therefore, the aim of this paper is to estimate the packing dimension and the proper-dimensional packing measure of each limit set in the family of the CIFSs of the generalized complex continued fractions, in order to improve the previous works \cite{IOS} and \cite{IS}. 

The precise statement is the following. 
Let 
\begin{align*}
    A_{0} := \{ \tau = u + iv \in \mathbb{C} \ | \ u \geq 0 \ \mathrm{and} \ v \geq 1 \} 
    \quad \text{and} \quad 
    X := \{ z \in \mathbb{C} \ | \ |z-1/2| \leq 1/2 \}, 
\end{align*}
and we set $I_{\tau} := \{ m +n\tau \in \mathbb{C} \ |\ m, n \in \mathbb{N} \}$ for each $\tau \in A_{0}$, where $\mathbb{N}$ is the set of the positive integers. 
\begin{definition}[The CIFS of generalized complex continued fractions]
Let $\tau \in A_{0}$. 
Then, we say that $S_{\tau} := \{ \phi_{b} \colon X \rightarrow X \ |\ b \in I_{\tau} \} $ is the CIFS of generalized complex continued fractions. 
Here, for each $\tau \in A_{0}$ and $b \in I_{\tau}$, $\phi_{b}$ is defined by 
\begin{align*}
    \phi_{b}(z) := \frac{1}{z + b} \quad ( z \in X ). 
\end{align*}
\end{definition}
We call $\{S_{\tau}\}_{\tau \in A_{0}}$ the family of CIFSs of generalized complex continued fractions. 
For each $\tau \in A_{0}$, let $J_{\tau}$ be the limit set of the CIFS $S_{\tau}$ (see Definition \ref{CIFSDEF}) and let $h_{\tau}$ be the Hausdorff dimension of the limit set $J_{\tau}$. 
We remark that this family of the CIFSs is a generalization of $\hat{S}$ in some sense. 
Indeed, $S_{\tau}$ is related to `generalized' complex continued fractions since each point of the limit set $J_{\tau}$ of $S_{\tau}$ is of the form
\begin{equation*}
\cfrac{1}{b_{1} + \cfrac{1}{b_{2} + \cfrac{1}{b_{3} + \cdots}}}
\end{equation*}
for some sequence $( b_{1}, b_{2}, b_{3}, \ldots) $ in $I_{\tau}$ (See Definition \ref{CIFSDEF}). 
Here, it should be mentioned that there are many kinds of general theories for continued fractions and related iterated function systems (\cite{IK}, \cite{MU}, \cite{MU2}, \cite{MiU}). 
In \cite{IS}, we showed the following theorem. 

\begin{theorem}[{\cite[Theorem 1.3]{IS}}]\label{main1}
Let $\{S_{\tau}\}_{\tau \in A_{0}}$ be the family of CIFSs of generalized complex continued fractions. Then, for each $\tau \in A_{0}$, we have $\mathcal{H}^{h_{\tau}}(J_{\tau}) = 0$ and $0 < \mathcal{P}^{h_{\tau}}(J_{\tau})$. 
\end{theorem}

We now present the main theorem in this paper. 
\begin{theorem}[the main theorem]\label{main2}
Let $\{S_{\tau}\}_{\tau \in A_{0}}$ be the family of CIFSs of generalized complex continued fractions. Then, for each $\tau \in A_{0}$, we have $\mathcal{P}^{h_{\tau}}(J_{\tau}) < \infty$. 
\end{theorem}
Combining Theorem \ref{main1} and Theorem \ref{main2}, we obtain the following corollary. 
\begin{corollary}\label{main}
Let $\{S_{\tau}\}_{\tau \in A_{0}}$ be the family of CIFSs of generalized complex continued fractions. Then, for each $\tau \in A_{0}$, we have $\mathcal{H}^{h_{\tau}}(J_{\tau}) = 0$ and $0 < \mathcal{P}^{h_{\tau}}(J_{\tau}) < \infty$. 
In particular, for each $\tau \in A_{0}$, the packing dimension of the limit set $J_{\tau}$ equals the Hausdorff dimension $h_{\tau}$ of the limit set $J_{\tau}$. 
\end{corollary}
\begin{remark}
By the general theory of finite CIFSs with the OSC, the proper-dimensional Hausdorff measure and the proper-dimensional packing measure of the limit set are positive and finite (see \cite{F1}, \cite{F2} and \cite{MU}). 
However, Corollary \ref{main} indicates that for each $S_{\tau}$ in the family of CIFSs of generalized complex continued fractions which consists of uncountably many elements, the packing dimension of the limit set equals the Hausdorff dimension of the limit set, the proper-dimensional Hausdorff measure of the limit set is zero, and the proper-dimensional packing measure of the limit set is positive and finite. 
This is the phenomenon which Mauldin and Urba\'{n}ski first discovered in \cite{MU} and cannot hold in the finite CIFSs under the OSC. 
\end{remark}
\begin{remark}
It was shown that for each $\tau \in A_{0}$, $\overline{J_{\tau }}\setminus J_{\tau}$ is at most countable and $h_{\tau }=\dim _{\mathcal{H}}(\overline{J_{\tau }}$) (\cite{IOS}). 
Thus, for each $\tau \in A_{0}$, we have $0 < \mathcal{P}^{h_{\tau}}(\overline{J_{\tau }}) = \mathcal{P}^{h_{\tau}}(J_{\tau }) < \infty$. 
Also, for each $\tau \in A_{0}$, since the set of attracting fixed points of elements of the semigroup generated by $S_{\tau}$ is dense in $J_{\tau }$, Theorem 1.1 of \cite{St} implies that $\overline{J_{\tau }}$ is equal to the Julia set of the rational semigroup generated by $\{ \phi _{b}^{-1}\mid b\in I_{\tau }\}$.  
\end{remark}
The ideas and strategies to prove the main theorem are the following. 
To prove the finiteness of the packing measure of the limit set $J_{\tau}$, we apply Lemma 4.10 in the paper \cite{MU} to $S_{\tau}$ for each $\tau \in A_{0}$. 
That is, it suffices to show that for each $r > 0$ (which is sufficiently small) and $b \in I_{\tau}$ with $\diam \phi_{b}(X) / r \ll 1$, we have
\begin{align} \label{image_of_goal}
    m_{\tau}(B(x, r)) \gtrsim r^{h_{\tau}},  
\end{align}
where $x := 1/b$, $B(x, r)$ is the open ball with the center $x$ and the radius $r$ with respect to the Euclidean distance in $\mathbb{C}$, $m_{\tau}$ is $h_{\tau}$-conformal measure of $S_{\tau}$ (see Theorem \ref{existenceofconformalmeasure}) and $f(r) \gtrsim g(r)$ means that there exists a `small' constant $c > 0$ such that $f(r) \geq c g(r)$ for all $r > 0$. 
Note that there is a useful inequality for the conformal measure $m_{\tau}$ : 
\begin{align} \label{useful_inequality_of_conformal_measure}
    m_{\tau}\left( \bigcup_{a \in I} \phi_{a}(X) \right)
    \gtrsim \sum_{a \in I} \ |a|^{-2h_{\tau}} 
\end{align}
for each $I \subset I_{\tau}$, where we use the property on the $S_{\tau}$ (see, Lemma \ref{basicestimate} and the inequality (\ref{useful_inequality})). 
To prove the above inequality (\ref{image_of_goal}), we essentially consider the following two cases:  
\begin{enumerate}
    \item $r \ll |x|$ and 
    \item $r \gg |x|$.  
\end{enumerate}
In the first case, by the assumptions, we deduce that $|x|^{2} \leq r$ and $ r < |x|$. 
We next define $I_{\tau, 1} \subset I_{\tau}$ and show the following inclusion and inequalities: 
\begin{align} \label{strategy_equation_inclusion_case1}
     B(x, r) \supset \bigcup_{a \in I_{\tau, 1}} \phi_{a}(X), \qquad
     |a| \lesssim |x|^{-1} \ \ (a \in I_{\tau, 1})
     \qquad \text{and} \qquad 
     \#(I_{\tau, 1}) \gtrsim \left( \frac{r}{|x|^{2}} \right)^{2}. 
\end{align}
Here, for any set $A$, we denote by $\#(A)$ the cardinality of $A$.  
Then, by the inequality in (\ref{useful_inequality_of_conformal_measure}) and the inclusion and inequalities in (\ref{strategy_equation_inclusion_case1}), we have
\begin{align*}
    m_{\tau}(B(x, r)) 
    &\geq m_{\tau}\left( \bigcup_{a \in I_{\tau, 1}} \phi_{a}(X) \right)
    \gtrsim \sum_{a \in I_{\tau, 1}} \ |a|^{-2h_{\tau}}
    \gtrsim \#(I_{\tau, 1}) \cdot |x|^{2h_{\tau}} \\
    &\gtrsim r^{2} \cdot |x|^{2h_{\tau}-4} 
    \gtrsim r^{2} \cdot r^{h_{\tau}-2} 
    =  r^{h_{\tau}}, 
\end{align*}
where we use the inequality $|x|^{2} \leq r$ and $h_{\tau} < 2$. 
Note that we prove some additional lemmas to prove the inclusion and inequalities in (\ref{strategy_equation_inclusion_case1}). 

In the second case, since $r \gg |x|$, we have
$B( 0, cr ) \subset B(x, r)$, where $c > 0$ is a `small' positive number. 
Next, we define $I_{\tau}(r) \subset I_{\tau}$ and show the following inclusion and inequalities: 
\begin{align} \label{strategy_equation_inclusion_case2}
    B(0, cr) \supset \bigcup_{a \in I_{\tau}(r)} \phi_{a}(X), \qquad
    |a| \lesssim r^{-1} \ \ (a \in I_{\tau}(r))
    \qquad \text{and} \qquad 
    \# (I_{\tau}(r)) \gtrsim r^{-2}. 
\end{align}
Then, by the inequality (\ref{useful_inequality_of_conformal_measure}) and the inclusion and  inequalities in (\ref{strategy_equation_inclusion_case2}), we have
\begin{align*}
    m_{\tau}(B(x, r)) 
    &\geq m_{\tau}(B(0, cr))
    \geq m_{\tau}\left( \bigcup_{a \in I_{\tau}(r)} \phi_{a}(X) \right)
    \gtrsim \sum_{a \in I_{\tau}(r)} \ |a|^{-2h_{\tau}} \\
    &\gtrsim \#(I_{\tau}(r)) \cdot r^{2h_{\tau}}
    \gtrsim r^{2h_{\tau}-2}
    > r^{h_{\tau}}, 
\end{align*}
where we use the inequalities $r^{2h_{\tau}-2} > r^{h_{\tau}}$ since $r$ is sufficiently small and $h_{\tau} < 2$.  

The rest of the paper is organized as follows. 
In Section 2, we summarize the general theory of the CIFSs and recall some definitions and theorems in the theory. 
In Section 3, we present some results for the CIFSs of generalized complex continued fractions in the paper \cite{IOS} and \cite{IS}. 
Also, we prove a slight modification of lemmas in the paper \cite{IOS} and \cite{IS} to prove the main theorem. 
In Section 4, we prove the main theorem (Theorem \ref{main2}). 
In this section, we first show some additional lemmas to prove the main theorem and next show the main theorem. 
Note that we consider three cases to prove the main theorem. 

\section{Conformal iterated function systems}
In this section, we summarize the general theory of CIFSs (\cite{IOS}, \cite{IS}, \cite{MU}, \cite{MU2}). 
We first recall the definition of CIFSs and the limit set of the CIFSs. 
\begin{definition}[Conformal iterated function system] \label{CIFSDEF}
    Let $X \subset \mathbb{R}^{d}$ be a non-empty compact and connected set with the Euclidean norm $|\cdot|$ and let $I$ be a finite set or bijective to $\mathbb{N}$. 
    Suppose that $I$ has at least two elements. 
    We say that $S := \{ \phi_{i} \colon X \to X \ |\ i \in I \}$ is a conformal iterated function system (for short, CIFS) if $S$ satisfies the following conditions.  
    \begin{enumerate}
        \item Injectivity: $\phi_{i} \colon X \to X$ is injective for each $i \in I$.
        \item Uniform Contractivity: There exists $c \in (0, 1) $ such that, for all $i \in I$ and $x, y \in X $, the following inequality holds: 
        \begin{align*}
            | \phi_{i}(x) - \phi_{i}(y) | \leq c| x - y |. 
        \end{align*}
        \item Conformality: There exist $\epsilon > 0$ and an open and connected subset $V \subset \mathbb{R}^{d}$ with $X \subset V$ such that for all $i \in I $, $\phi_{i} $ extends to a $C^{1+\epsilon}$diffeomorphism on $V$ and $\phi_{i} $ is conformal on $V$ 
        i.e. for each $x \in V$ and $i \in I$, there exists $C_{i}(x) > 0$ such that for each $u, v \in \mathbb{R}^{d}$, 
        \begin{align*}
            |\phi_{i}^{\prime}(x)u - \phi_{i}^{\prime}(x)v| = C_{i}(x)|u - v|. 
        \end{align*}
        Here, $\phi_{i}^{\prime}(x)$ denotes the derivative of $\phi_{i}$ at $x \in V$. 
        \item Open Set Condition (OSC): For all $ i, j \in I \ ( i \neq j )$, $ \phi_{i}(\inter(X)) \subset \inter(X)$ and  $ \phi_{i}(\inter(X)) \cap  \phi_{j}(\inter(X)) = \emptyset $. Here, $\inter(X)$ denotes the set of interior points of $X$ with respect to the topology in $\mathbb{R}^{d}$.
        \item Bounded Distortion Property(BDP): There exists $\tilde{K} \geq 1 $ such that for all $x, y \in V $ and for all $w \in I^{*} := \bigcup_{n=1}^{\infty} I^{n}$, the following inequality holds:
        \begin{align*}
            |\phi^{\prime}_{w}(x)| \leq \tilde{K} \cdot |\phi^{\prime}_{w}(y)|. 
        \end{align*}
        Here, for each $n \in \mathbb{N}$ and $w = w_{1}w_{2} \cdots w_{n} \in I^{n}$, we set $\phi_{w} := \phi_{w_{1}} \circ \phi_{w_{2}} \circ \cdots \circ \phi_{w_{n}}$ and $\displaystyle |\phi^{\prime}_{w}(x)|$ denotes the  norm of the derivative of $\phi_{w}$ at $x \in X$ with respect to the Euclidean norm on $\mathbb{R}^{d}$.
        \item Cone Condition: For all $x \in \partial X$, there exists an open cone $\Con(x, u, \alpha)$ with a vertex $x$, a direction $u$, an altitude $|u|$ and an angle $\alpha$ such that $\Con(x, u, \alpha)$ is a subset of $\inter(X)$.
    \end{enumerate}
\end{definition}
We endow $I$ with the discrete topology and endow $I^{\infty} := I^{\mathbb{N}}$ with the product topology.
Note that $I^{\infty}$ is Polish in general and $I^{\infty}$ is a compact metrizable space if $I$ is a finite set.

Let $S$ be a CIFS and we set $w|_{n} := w_{1}w_{2} \cdots w_{n} \in I^{n}$ and $\phi_{w|_{n}} := \phi_{w_{1}} \circ \phi_{w_{2}} \circ \cdots \circ \phi_{w_{n}}$ for each $w=w_{1}w_{2}w_{3} \cdots \in I^{\infty}$. 
Note that
$\bigcap_{n \in \mathbb{N}} \phi_{w|_{n}}(X)$ is a singleton (denoted by $\{ x_{w} \}$) %, 
and 
the coding map $\pi_{S} \colon I^{\infty} \rightarrow X$ of $S$ defined by $\pi_{S}(w) := x_{w}$ is well-defined. 
Then, the limit set $J_{S}$ of $S$ is defined by
\begin{align*}
    J_{S} :=\pi(I^{\infty}) = \bigcup_{w \in I^{\infty}} \bigcap_{n \in \mathbb{N}} \phi_{w|_{n}}(X) ( \subset X \subset \mathbb{R}^{d} ).  
\end{align*}
We set $h_{S} := \dim_{\mathcal{H}}J_{S}$, where we denote by $\dim_{\mathcal{H}} A$ the Hausdorff dimension of a set $A \subset \mathbb{R}^{d}$ with respect to the Euclidean distance.

We next recall the pressure function of CIFS $S$ as follows in order to define the regularity of CIFSs. 
\begin{definition}\label{PSD}
    For each $n\in \mathbb{N}$, $[0, \infty]$-valued function $\psi^{n}_{S}$ is defined by
    \begin{align*}
        \psi^{n}_{S}(t) := \sum_{w \in I^{n}} ||\phi_{w}^{\prime}||_{X}^{t} \quad (t \geq 0). 
    \end{align*}
    Here, for a $C^{1}$ map $f \colon Y \to \mathbb{R}^{d}$ ($Y \subset \mathbb{R}^{d}$), we set
    \begin{align*}
        |f^{\prime}(y)| := \sup \{ |f^{\prime}(y) u | \ | \ u \in \mathbb{R}^{d}, |u| = 1 \} \ (y \in Y) \quad \text{and} \quad 
        ||f^{\prime}||_{Y} := \sup\{ | f^{\prime}(y) | \ | \ y \in Y \}. 
    \end{align*}
\end{definition}
We set $\theta_{S} := \inf\{ t \geq 0 |\ \psi_{S}^{1}(t) < \infty \}$ and $F(S) := \{ t \geq 0 \ | \ \psi^{1}_{S}(t) < \infty \}$. 
%Note that by the following lemma, we deduce that $F(S) = (\theta_{S}, \infty)$ or $F(S) = [ \theta_{S}, \infty)$. 
%\begin{lemma}[\cite{MU}]
%    Let $S$ be a CIFS. 
%    Then, $\psi_{S}^{1}(t)$ is non-increasing on $[0, \infty)$, and decreasing and convex on $F(S)$. 
%    In addition, we have $\psi_{S}^{1}(d) \leq \tilde{K}^d$. 
%    In particular, $\theta_{S} \leq d$. 
%\end{lemma}
%In addition, by the following proposition, we deduce the basic properties of $\psi_{S}^{n}$.  
%\begin{proposition}[\cite{MU}]\label{logsubadditive}
%    Let $S$ be a CIFS. 
%    For all $m, n \in \mathbb{N}$ and $t \geq 0$, we have $$\tilde{K}^{-2t} \psi_{S}^{k}(t)\psi_{S}^{n}(t) \leq \psi_{S}^{m+n}(t) \leq \psi_{S}^{m}(t)\psi_{S}^{n}(t). $$
%\end{proposition}
%In particular, $\psi_{S}^{n}(t) < \infty$ for each $n \in \mathbb{N}$ if and only if $\psi_{S}^{n}(t) < \infty$ for some $n \in \mathbb{N}$ (or $n=1$), 
%and the function $n \mapsto \log\psi_{S}^{n}(t)$ is subadditive for all $t \geq 0$.
%By the subadditivity of $\log\psi_{S}^{n}(t)$, we now define the pressure function of $S$ as follows. 
Note that $\log\psi_{S}^{n}(t)$ is subadditive with respect to $n \in \mathbb{N}$ for each $t \geq 0$ (see \cite{MU}). 
By the subsadditivity, we now define the pressure function of $S$ as follows. 
\begin{definition}[Pressure function] \label{PSfunc}
    The pressure function of $S$ is the function $P_{S} \colon [0, \infty) \to (-\infty, \infty]$ defined by 
    \begin{align*}
        \displaystyle P_{S}(t) := \lim_{n \to \infty} \frac{1}{n} \log \psi^{n}_{S}(t) \in (-\infty, \infty] \quad (t \geq 0). 
    \end{align*}
\end{definition}
%\begin{proposition}[\cite{MU}]
%    Let $S$ be a CIFS and $P_{S}$ be the pressure function of $S$. 
%    Then, $P_{S}(t) < \infty$ if and only if $\psi_{S}^{1}(t) < \infty$ for each $t \geq 0$. 
%    In particular, $\theta_{S} = \inf\{ t \geq 0 \ | \ P_{S}(t) < \infty \}$. 
%    In addition, 
%    $P_{S}$ is non-increasing on $[0, \infty)$, and decreasing and convex on $F(S)$. 
%\end{proposition}
%Note that $P_{S}(0) = \infty$ if and only if $I$ is infinite. 
By using the pressure function in Definition \ref{PSfunc}, we define the regularity of CIFSs.
\begin{definition}[Regularity and Hereditarily regularity of CIFS]\label{RD}
    Let $S$ be a CIFS. We say that 
    \begin{itemize}
        \item $S$ is regular if there exists $t \geq 0$ such that $P_{S}(t) = 0$, and 
%        \item $S$ is strongly regular if there exists $t \geq 0$ such that $P_{S}(t) \in (0, \infty)$ and 
        \item $S$ is hereditarily regular if, for all $I^{\prime} \subset I$ with $\#(I \setminus I^{\prime}) < \infty $, $S^{\prime} := \{ \phi_{i} \colon X \to X \ |\ i \in I^{\prime} \} $ is regular. 
    \end{itemize}
    Here, for any set $A$, we denote by $\#(A)$ the cardinality of $A$.
\end{definition}
%Note that if a CIFS $S$ is hereditarily regular then $S$ is strong regular, and if $S$ is strong regular then $S$ is regular. 
Note that if a CIFS $S$ is hereditarily regular then $S$ is regular (see \cite{MU}). 

We finally recall the $h_{S}$-conformal measure of $S$. 
If a CIFS $S$ is regular, there is the following `nice' probability measure $m_{S}$ ($h_{S}$-conformal measure of $S$) on $J_S$. 
Indeed, we often use $m_{S}$ in order to estimate 
the packing measure of the limit set of CIFSs. 
\begin{theorem}[\cite{MU} Lemma 3.13] \label{existenceofconformalmeasure}
Let $S$ be a CIFS. 
If $S$ is regular, then there exists the unique Borel probability measure $m_{S}$ on $X$ such that the following properties hold. 
\begin{enumerate}
\item $m_{S}(J_{S}) = 1$. 
\item For all Borel subset $A$ on $X$ and $i \in I$, $m_{S}(\phi_{i}(A)) = \int_{A}|\phi_{i}^{\prime}(y)|^{h_{S}} m_{S} (\mathrm{d} y) $. 
\item For all $i, j \in I$ with $i \neq j$, $m_{S}(\phi_{i}(X) \cap \phi_{j}(X) ) = 0$. 
\end{enumerate}
\end{theorem}
We call $m_{S}$  the  $h_{S}$-conformal measure of $S$. 
As we mentioned above, by the existence of the conformal measure of CIFSs, we estimate the packing measure and obtain the following key theorem to prove Theorem \ref{main2}. 

\begin{theorem}[\cite{MU} Lemma 4.10]\label{packingmeasurefiniteness}
Let $S$ be a regular CIFS and $m_{S}$ be the $h_{S}$-conformal measure of $S$. 
Suppose that there exist $L >0$, $\xi > 0$ and $\gamma \geq 1$ such that for all $b \in I$ and $r > 0$ with $\gamma \cdot \diam \phi_{b}(X) \leq r \leq \xi$, there exists $x_{0} \in \phi_{b}(V)$ such that $m_{S}(B(x_{0}, r)) \geq L r^{h_{S}}$, where $B(x_{0}, r)$ is the open ball with the center $x_{0}$ and the radius $r$ with respect to the Euclidean distance in $\mathbb{R}^{d}$. 
Then, we have $\mathcal{P}^{h_{S}}(J_{S}) < \infty$. 
\end{theorem}
\section{CIFSs of generalized complex continued fractions}
In this section, we present some results on the CIFSs of generalized complex continued fractions introduced in the papers \cite{IOS} and \cite{IS}, which are needed to prove the results of this paper. 
Note that these CIFSs are interesting examples of infinite CIFSs. 
Rest of this paper, we denote by $B(y, r) \subset \mathbb{R}^{d} (d \in \mathbb{N})$ the open ball with center $y \in \mathbb{R}^{d}$ and radius $r > 0$, with respect to the $d$-dimensional Euclidean norm and we identify $\mathbb{C}$ with $\mathbb{R}^{2}$. 

We first present the following lemma shown in \cite{IOS} and \cite{IS} in order to prove Theorem \ref{main2}. 
\begin{lemma}[Lemmas 3.1, 3.3 and 3.4 in \cite{IS}]\label{GCCFisCIFS}
    For all $\tau \in A_{0}$, $S_{\tau}$ is a hereditarily regular CIFS. 
    In addition, we have $1 < h_{\tau} < 2$. 
\end{lemma}
In addition, in order to prove Theorem \ref{main2}, we next %present and 
prove the following lemma (a slight modification of Lemma 3.2 in \cite{IS}). 
For the readers, we give a proof of Lemma \ref{basicestimate}. 
\begin{lemma}\label{basicestimate}
Let $\tau \in A_{0}$. 
Then, there exists $K_{0} \geq 1$ such that for all $K \geq K_{0}$ and $a \in I_{\tau}$, the following properties hold. 
\begin{enumerate}
    \item $\phi_{a}(X) \subset B(0, K |a|^{-1})$. 
    \item $K^{-1} |a|^{-2} \leq |\phi_{a}^{\prime}(z)| \leq K |a|^{-2}$ for each $z \in X$. 
    \item $K^{-1} |a|^{-2} \leq \diam{\phi_{a}(X)}$. 
\end{enumerate}
\end{lemma}
\begin{proof}
    Let $\tau \in A_{0}$. 
    Note that by using the BDP, there exists a constant $C \geq 1$ such that for all $z, w \in X$, 
    \begin{align} \label{basicinequalityBDP}
        |\phi^{\prime}_{a}(z)| \leq C \cdot |\phi^{\prime}_{a}(w)|. 
    \end{align}
    We set $K_{0} := C ( \geq 1)$ and let $K \geq K_{0}$ and $a \in I_{\tau}$. 
    Then, by the inequality (\ref{basicinequalityBDP}) with $w = 0 \in X$, we have
    \begin{align*} 
        &|\phi_{a}(z)| \cdot |a| 
        = \frac{|a|}{|a+z|} 
        = \sqrt{\frac{|a|^{2}}{|a+z|^{2}}}
        = \sqrt{\frac{|\phi_{a}^{\prime}(z)|}{|\phi_{a}^{\prime}(0)|}}
        \leq \sqrt{C} 
        \leq K_{0} 
        \leq K 
    \end{align*}
    for each $z \in X$. 
    It follows that $\phi_{a}(X) \subset B(0, K |a|^{-1})$. 
    Also, by the inequality (\ref{basicinequalityBDP}), we have
    \begin{align*}
        & K^{-1} |a|^{-2}
        = K^{-1} |\phi_{a}^{\prime}(0)| 
        \leq C^{-1}|\phi_{a}^{\prime}(0)| 
        \leq |\phi_{a}^{\prime}(z)| \quad \text{and} \quad |\phi_{a}^{\prime}(z)| 
        \leq C |\phi_{a}^{\prime}(0)|
        \leq K |\phi_{a}^{\prime}(0)|  
        = K |a|^{-2}
    \end{align*}
    for each $z \in X$, 
    which deduce that $K^{-1} |a|^{-2} \leq |\phi_{a}^{\prime}(z)| \leq K |a|^{-2}$. 
Moreover, by the inequality (\ref{basicinequalityBDP}), we have 
\begin{align*}
    \diam{\phi_{a}(X)} \cdot |a|^{2}
    &\geq |\phi_{a}(z) - \phi_{a}(w)| \cdot |a|^{2}
    = \left| \frac{1}{z+a} - \frac{1}{w+a} \right| \cdot |a|^{2}
    =  \frac{|w - z|}{|z+a||w+a|} \cdot |a|^{2} \\
    &= |w - z| \cdot \sqrt{\frac{|a|^{2}}{|a+z|^{2}}} \cdot \sqrt{\frac{|a|^{2}}{|a+w|^{2}}}
    = |w - z| \cdot  \sqrt{\frac{|\phi_{a}^{\prime}(z)|}{|\phi_{a}^{\prime}(0)|}} \cdot \sqrt{\frac{|\phi_{a}^{\prime}(w)|}{|\phi_{a}^{\prime}(0)|}}
    \geq |w - z| \cdot C^{-1}
\end{align*}
for all $z, w \in X = B(1/2, 1/2)$. 
Since $\diam X = \sup\{|z-w| \ | \ z, w \in X\} = 1$, we obtain that $ \diam{\phi_{a}(X)} \geq C^{-1} |a|^{-2} \geq K^{-1} |a|^{-2}$. 
Therefore, we have proved our lemma. 
\end{proof}
Note that by Theorem \ref{existenceofconformalmeasure} and Lemma \ref{basicestimate}, for each $I \subset I_{\tau}$, we have
\begin{align}
    m_{\tau}\left( \bigcup_{a \in I} \phi_{a}(X) \right)
    &= \sum_{a \in I} m_{\tau}\left( \phi_{a}(X) \right)
    = \sum_{a \in I} \int_{X}|\phi_{a}^{\prime}(y)|^{h_{\tau}} m_{\tau} (\mathrm{d} y) \notag \\
    &\geq \sum_{a \in I} K^{-h_{\tau}} |a|^{-2h_{\tau}} m_{\tau}(X)
    = \sum_{a \in I} K^{-h_{\tau}} |a|^{-2h_{\tau}}. \label{useful_inequality}
\end{align}

Rest of this section, we recall notations and results used in the paper \cite{IS}. 
We identify $I_{\tau}$ with $\{ {}^{t}(s, t) \in \mathbb{R}^{2} \ | \ s+it \in I_{\tau} \}$ and $\mathbb{N}^{2}$ with $\{ {}^{t} (m, n) \in \mathbb{R}^{2} \ | \ m, n \in \mathbb{N} \}$, where for any matrix $A$, we denote by ${}^{t} A$ the transpose of $A$.  
For each $\tau = u + iv \in A_{0}$, we set 
\begin{align*}
    E_{\tau} := \left( \begin{array}{cc} 1 & u \\ 0 & v \end{array}\right) \quad \text{and} \quad F_{\tau} := {}^{t}E_{\tau} E_{\tau} = \left( \begin{array}{cc} 1 & u \\ u & |\tau|^{2} \end{array}\right).     
\end{align*}
Note that by direct calculations, $E_{\tau}\mathbb{N}^{2} = I_{\tau}$, $E_{\tau}$ is invertible and there exist the eigenvalues $\lambda_{1} > 0$ and $\lambda_{2} > 0$ of $F_{\tau}$ with $\lambda_{1} < \lambda_{2}$. 
Note that since $F_{\tau}$ is symmetric, there exist an eigenvector $v_{1} \in \mathbb{R}^{2}$ of $F_{\tau}$ with respect to $\lambda_{1}$ and an eigenvector $v_{2} \in \mathbb{R}^{2}$ of $F_{\tau}$ with respect to $\lambda_{2}$ such that $V_{\tau} := (v_{1}, v_{2})$ is an orthogonal matrix. 

For each $\tau \in A_{0}$, we set $N_{\tau} := \sqrt{2\lambda_{2}}/\sqrt{\lambda_{1}} + 1 \ (> 2)$. 
In addition, for each $\tau \in A_{0}$ and $l > 0$, we set
\begin{align*}
&D_{1}(\tau, l) := \{ ^{t}(x, y) \in \mathbb{R}^{2} \ | \ l^{2}/\lambda_{1} < x^{2} + y^{2} \leq (N_{\tau} l)^{2}/\lambda_{2} \} \quad \text{and} \\
& D_{2}(\tau, l) := \{ ^{t}(x, y) \in \mathbb{R}^{2} \ |\ l^{2} < x^{2} + y^{2} \leq (N_{\tau} l)^{2} \}. 
\end{align*}
Note that $l/\sqrt{\lambda_{1}} < (N_{\tau}l)/\sqrt{\lambda_{2}}$ for each $l > 0$ since $\sqrt{\lambda_{2}}/\sqrt{\lambda_{1}} < N_{\tau}$ and $l < N_{\tau} l$ for each $l > 0$ since $1 < N_{\tau}$. 

By these notations, we present the result in \cite{IS} (Lemma \ref{latticepointestimate}) and its slight modification (Lemma \ref{lattice_point_in_annulus}). 
For the readers we give a proof of Lemma \ref{lattice_point_in_annulus}. 
Note that Lemma \ref{lattice_point_in_annulus} is used in Case (3) in the proof of Theorem \ref{main2}. 
\begin{lemma}[Lemma 4.4 in \cite{IS}] \label{latticepointestimate}
    Let $\tau \in A_{0}$. 
    Then, there exist $\tilde{L}_{\tau}>0$ and $\tilde{C}_{\tau} > 0$ such that for all $l > \tilde{C}_{\tau}$, 
    \begin{align*}
        \#(I_{\tau} \cap D_{2}(\tau, l))
        \geq \tilde{L}_{\tau} l^{2} -  \frac{7N_{\tau}}{2\sqrt{\lambda_{2}}} l. 
    \end{align*}
\end{lemma}
\begin{lemma} \label{lattice_point_in_annulus}
    Let $\tau \in A_{0}$. 
    Then, there exist $Q_{\tau} > 0$ and $C_{\tau} > 0$ such that for all $l \geq C_{\tau}$, we have 
    \begin{align}\label{latticepointessentialestimate}
        \#(I_{\tau} \cap D_{2}(\tau, l)) >  Q_{\tau} l^{2}. 
    \end{align}
\end{lemma}
\begin{proof}
    Let $\tau \in A_{0}$ and we set $Q_{\tau} := \tilde{L}_{\tau}/2 > 0$, where $\tilde{L}_{\tau} > 0$ is the number in Lemma \ref{latticepointestimate}. 
    Note that there exists $C_{\tau} > \tilde{C}_{\tau}$ such that for all $l \geq C_{\tau}$, $ ( \tilde{L}_{\tau} l^{2})/2  -  (7 N_{\tau} l)/(2\sqrt{\lambda_{2}}) > 0$, which is equivalent to 
    \begin{align} \label{sufficientlarge_positive}
        \tilde{L}_{\tau} l^{2} -  \frac{7N_{\tau}}{2\sqrt{\lambda_{2}}} l > \frac{\tilde{L}_{\tau}}{2} l^{2},  
    \end{align}
    where $\tilde{C}_{\tau} > 0$ is the number in Lemma \ref{latticepointestimate}. 
    By Lemma \ref{latticepointestimate} and the inequality (\ref{sufficientlarge_positive}), we deduce that
    \begin{align*}
        \#(I_{\tau} \cap D_{2}(\tau, l)) 
        \geq \tilde{L}_{\tau} l^{2} -  \frac{7N_{\tau}}{2\sqrt{\lambda_{2}}} l 
        > \frac{\tilde{L}_{\tau}}{2} l^{2} = Q_{\tau} l^{2}
    \end{align*}
    for all $l \geq C_{\tau}$. 
    Therefore, we have proved our lemma. 
\end{proof}
%We finally present the following proposition in \cite{IS} and its corollary to prove the lemma \ref{latticepointcase1lemma}. 
%\begin{proposition}[Proposition 4.3 in \cite{IS}] \label{latticepointlemma}
%Let $l > 0$. Then, for each $l \geq 6$, 
%\begin{align*}
%    0 < \frac{l^{2}-7l+7}{2} \leq \#(\{ ^{t}(m, n) \in \mathbb{N}^{2} \ | \ m^{2} + n^{2} \leq l^{2} \}) \leq l^{2}. 
%\end{align*}
%\end{proposition}
%Note that, by Proposition \ref{latticepointlemma}, we deduce that if $l>15$, then
%\begin{align} \label{latticepointlemma_simple}
%    \#(\{ ^{t}(m, n) \in \mathbb{N}^{2} \ | \ m^{2} + n^{2} \leq l^{2} \}) > \frac{l^{2}}{4}. 
%\end{align}
%Indeed, if $l >15$, then we have $(l^{2} -7l +7)/2 -l^{2}/4 \geq (l^{2} -14l)/4 > 0$. 
%
%
%
%
%
%
%
%
%
%
%
%
%
%
%
%
%
%
%
%
%
%
%
%
%
%
%
%
%
%
%
%
%
%
%
%
%
%
%
%
%
%
%
%
%
%
%
%
%
%
%
%
%
%
%
%
%
%
%
%
%
%
%
%
%
%
%
%
%
%
%
%
%
%
%
%
%
%
%
%
%
%
%
%
%
%
%
%
%
%
%
%
%
%
%
%
%
%
%
%
%
%
%
%
%
%
%
%
%
%
%
%
%
%
%
%
%
%
%
%
%
%
%
%
%
%
%
%
%
%
%
%
%
%
%
%
%
%
%
%
%
%
%
%
\section{Proof of the main theorem}
In this section, we prove the main theorem (Theorem \ref{main2}). 
Rest of this paper, we use the notations in Section 3. 
\subsection{Lemmas for the main theorem} \label{lemmas_for_main_theorem}
In this subsection, we prove the following lemmas to apply to Case (1) in the proof of Theorem \ref{main2}. 

\begin{lemma}
    Let $f(z) := 1/z \ (z \in \mathbb{C} \setminus \{ 0 \})$. 
    Then, for each $B(x, r) (\subset \mathbb{C})$ with $r < |x|$, we have 
    \begin{align*}
        f(B(x, r)) = B\left( \frac{|x|^{2}}{|x|^{2} - r^{2}}\cdot \frac{1}{x}, \ \frac{r}{|x|^{2} - r^{2}} \right). 
    \end{align*}
\end{lemma}
\begin{proof} \label{inverse_circle}
    Let $z \in \partial B(x, r)$. 
    Since $|z-x| = r$ for each $z \in \partial B(x, r)$, we have
    \begin{align*}
        \left| \frac{1}{z} - \frac{\overline{x}}{|x|^{2} - r^{2}} \right|
        &= \left| \frac{x \overline{x} -r^{2} - \overline{x}z}{z(|x|^{2} - r^{2})} \right|
        = \left| \frac{|x - z|^{2} - r^{2} - \overline{z}(x-z)}{z(|x|^{2} - r^{2})} \right|
        = \frac{|\overline{z}||x-z|}{|z|(|x|^{2} - r^{2})}
        = \frac{r}{|x|^{2} - r^{2}}. 
    \end{align*}
    It follows that $f(\partial B(x, r)) = \partial B( \bar{x}/(|x|^{2} - r^{2}), r/(|x|^{2} - r^{2}))$. 
    Besides, since $f(x) = 1/x$ and $r < |x|$, we have
    \begin{align*}
        \left| f(x) - \frac{|x|^{2}}{|x|^{2} - r^{2}}\cdot \frac{1}{x} \right| = \frac{r}{|x|^{2} - r^{2}} \frac{r}{|x|} < \frac{r}{|x|^{2} - r^{2}}. 
    \end{align*}
    It follows that $f(x) \in B( \bar{x}/(|x|^{2} - r^{2}), r/(|x|^{2} - r^{2}))$. 
    Therefore, we have proved our lemma. 
\end{proof}

\begin{lemma}\label{Etauestimate}
    Let $\tau \in A_{0}$, $z_{0} \in \mathbb{R}^{2}$ and $R_{0} > 0$. 
    Then, we have
    \begin{align*}
        E_{\tau}(B(E_{\tau}^{-1} z_{0}, R_{0}/\sqrt{\lambda_{2}})) \subset B(z_{0}, R_{0}). 
    \end{align*}
\end{lemma}
\begin{proof}
    Let $z \in B(E_{\tau}^{-1} z_{0}, R_{0}/\sqrt{\lambda_{2}})$. 
    We set $\Delta z := z - E_{\tau}^{-1} z_{0}$ for simplicity. 
    Since $V_{\tau}$ is orthogonal and $|\Delta z|^{2} = | z - E_{\tau}^{-1} z_{0} |^{2} < R_{0}^{2}/\lambda_{2}$, we have 
    \begin{align*}
        | E_{\tau} z - z_{0} |^{2} 
        &= |E_{\tau} (\Delta z)|^{2} 
        = {}^{t} (\Delta z) {}^{t} E_{\tau} E_{\tau} \Delta z 
        = {}^{t} (\Delta z) F_{\tau} (\Delta z) \\
        &= {}^{t} (\Delta z) V_{\tau} \left( \begin{array}{cc} \lambda_{1} & 0 \\ 0 & \lambda_{2} \end{array}\right) {}^{t}V_{\tau} (\Delta z) 
        = \lambda_{1} v_{1}^{2} + \lambda_{2} v_{2}^{2}
        \leq \lambda_{2} |{}^{t} V_{\tau} (\Delta z)|^{2} 
        = \lambda_{2} |\Delta z|^{2} 
        < R_{0}^{2},    
    \end{align*}
    where ${}^{t} (v_{1}, v_{2}) := {}^{t} V_{\tau} (\Delta z )$. 
    Therefore, we have proved our lemma. 
\end{proof}
By %inequality (\ref{latticepointlemma_simple}) and 
Lemma \ref{Etauestimate}, we finally show the following lemma. 
\begin{lemma}\label{latticepointcase1lemma}
    Let $\tau \in A_{0}$. 
    Then, there exist $C_{\tau}^{\prime} \geq 1$ and $Q_{\tau}^{\prime} > 0$ such that for all
    $w \in E_{\tau}(\mathbb{R}_{+}^{2})$ and $T \geq C_{\tau}^{\prime}$ with $|w| > T$, we have 
    \begin{align*}
        \# (I_{\tau} \cap B(0, |w|) \cap B(w, T)) > Q_{\tau}^{\prime} T^{2}, 
    \end{align*}
    where $\mathbb{R}_{+}$ is the set of positive real numbers. 
\end{lemma}
\begin{proof}
Let $\tau \in A_{0}$. 
We set $Q_{\tau}^{\prime} := 1/(32\lambda_{2})$ and let $C_{\tau}^{\prime} \geq 34 \sqrt{\lambda_{2}}+1$ be a number  such that, for each $T \geq C_{\tau}^{\prime}$,  
\begin{align}\label{halfcoefficient2}
    \frac{(T - 2\sqrt{2\lambda_{2}})^{2}}{16 \lambda_{2}} 
    - \frac{T^{2}}{32\lambda_{2}} > 0. 
\end{align}
Let $T \geq C_{\tau}^{\prime}$. 
Note that $l := T/(2\sqrt{\lambda_{2}}) -\sqrt{2} > 17 - 2 = 15$, and by using a geometric observation 
%and the inequality (\ref{latticepointlemma_simple})
(and Proposition 4.3 in \cite{IS}), we deduce that for each $m_{0}, n_{0} \in \mathbb{N}$
\begin{align} \label{laticepoint_revised}
    \# (\{ ^{t}(m, n) \in \mathbb{N}^{2} \ | \ (m - m_{0})^{2} + (n - n_{0})^{2} \leq l^{2} \})
    &= \# ( \{ ^{t}(m, n) \in \mathbb{N}^{2} \ | \ m^{2} + n^{2} \leq l^{2} \}) 
    %\notag \\  &
    > \frac{l^{2} -7l +7}{2} 
    > \frac{l^{2}}{4}.  
\end{align}

Now, let $w \in E_{\tau}(\mathbb{R}_{+}^{2})$ with $|w| > T$ and we set $\xi := E_{\tau}^{-1} \left( 1 - T/(2|w|) \right) w \in \mathbb{R}_{+}^{2}$. 
Then, note that $B\left( E_{\tau} \xi, T/2 \right) \subset B(0, |w|) \cap B(w, T)$. 
Indeed, let $z \in B\left( E_{\tau} \xi, T/2 \right)$. 
Since $|w| > T$ and $E_{\tau} \xi = w - T/(2|w|) w$, we deduce the following inequalities: 
\begin{align*}
    &|z| 
    \leq \left| z - \left( w - \frac{T}{2|w|} w \right) \right| + \left|  w - \frac{T}{2|w|} w \right|  
    < \frac{T}{2} + \left|  1 - \frac{T}{2|w|} \right||w| 
    = \frac{T}{2} + \left( 1 - \frac{T}{2|w|} \right) |w| 
    = |w| \quad \text{and} \\  
    &|z - w| \leq \left| z - \left( w - \frac{T}{2|w|}w \right) \right| + \left| \frac{T}{2}\frac{w}{|w|} \right| 
    < \frac{T}{2} + \frac{T}{2} 
    = T.  
\end{align*}

Therefore, by Lemma \ref{Etauestimate} with $z_{0} := E_{\tau} \xi$ and $R_{0} := T/2$, we have 
\begin{align} \label{small_circle_estimate}
    E_{\tau} \left( \mathbb{N}^{2} \cap B \left( \xi, \frac{T}{2\sqrt{\lambda_{2}}} \right) \right)
    = I_{\tau} \cap E_{\tau}B \left( \xi, \frac{T}{2\sqrt{\lambda_{2}}} \right)
    \subset I_{\tau} \cap B\left( E_{\tau} \xi, \frac{T}{2} \right) 
    \subset I_{\tau} \cap B(0, |w|) \cap B(w, T), 
\end{align} 
where we use the fact $I_{\tau} = E_{\tau}(\mathbb{N}^{2})$ and $E_{\tau}$ is injective. 
In addition, let $\zeta_{1}$ and $\zeta_{2}$ be minimum integers with $\zeta_{1} \geq \xi_{1}$ and $\zeta_{2} \geq \xi_{2}$, where $\xi := (\xi_{1}, \xi_{2})$. 
We set $\zeta := {}^{t}(\zeta_{1}, \zeta_{2}) \in \mathbb{N}^{2}$. 
Then, since $|\zeta - \xi|^{2} \leq 2$, we have
\begin{align} \label{ajustlattice}
    B \left( \zeta, \frac{T}{2\sqrt{\lambda_{2}}} - \sqrt{2} \right) \subset B \left( \xi, \frac{T}{2\sqrt{\lambda_{2}}} \right). 
\end{align} 
By the inclusions (\ref{small_circle_estimate}) and (\ref{ajustlattice}), the inequalities (\ref{laticepoint_revised}) and (\ref{halfcoefficient2}), and the definition of $C_{\tau}^{\prime}$ and $Q_{\tau}^{\prime}$, we have
\begin{align*}
    \# ( I_{\tau} \cap B(0, |w|) \cap B(w, T))
    &\geq \# \left(\mathbb{N}^{2} \cap B \left( \xi , \frac{T}{2\sqrt{\lambda_{2}}} \right) \right) 
    \geq \# \left( \mathbb{N}^{2} \cap B \left( \zeta, \frac{T}{2\sqrt{\lambda_{2}}} - \sqrt{2} \right) \right) \\
    &= \# ( \{ ^{t}(m, n) \in \mathbb{N}^{2} \ | \ (m - \zeta_{1})^{2} + (n - \zeta_{2})^{2} \leq l^{2} \} ) \\
    &> \frac{l^{2}}{4} 
    = \frac{(T - 2\sqrt{2\lambda_{2}})^{2}}{16\lambda_{2}}
    > \frac{T^{2}}{32\lambda_{2}}
    = Q_{\tau}^{\prime} T^{2}.   
\end{align*}
Thus, we have proved our lemma. 
\end{proof}
Note that by replacing $C_{\tau}^{\prime}$ a sufficient large number, we also obtain that $(T-1)/T \geq 1/2$ for each $T \geq C_{\tau}^{\prime}$. 
\subsection{Proof for Theorem \ref{main2}}
\label{setting_for_main_theorem}
We now prove Theorem \ref{main2}. 
Note that $S_{\tau}$ is regular for each $\tau \in A_{0}$ by Lemma \ref{GCCFisCIFS} and the comment after Definition \ref{RD}. 
Rest of this section, let $K \geq 1$ be a number which satisfies 1. $\sim$ 3. in Lemma \ref{basicestimate}. 

\begin{proof}[Proof of Theorem \ref{main2}]
It suffices to show that $S_{\tau}$ satisfies the assumption of Theorem \ref{packingmeasurefiniteness} for each $\tau \in A_{0}$. 
Let $\tau \in A_{0}$ and we set $r_{0} := \min \{ 1/8, K C_{\tau}^{-1} \} (> 0)$, where $C_{\tau} > 0$ is the number in Lemma \ref{lattice_point_in_annulus}. 
Recall that $N_{\tau} = \sqrt{2\lambda_{2}}/\sqrt{\lambda_{1}} + 1 (> 2)$, and $C_{\tau}^{\prime}$, $Q_{\tau}^{\prime}$ and $Q_{\tau}$ are the numbers in Lemmas \ref{latticepointcase1lemma} and \ref{lattice_point_in_annulus} respectively (also, recall the comment after the proof of Lemma \ref{latticepointcase1lemma}). 
We define constants $L_{\tau}^{\prime}$ and $L_{\tau}$ as follows: 
\begin{align*}
    L^{\prime}_{\tau} 
    &:= \min \{ Q_{\tau}^{\prime}/4, (C_{\tau}^{\prime} + 1)^{-2} \} \ (> 0)%, \\
    %\xi 
    %&:= r_{0}^{2} (> 0),  \\
    %\gamma 
    %&:= K (\geq 1)
    \quad \text{and} \\
    L_{\tau} 
    &:= \min \left\{L_{\tau}^{\prime}( 8K )^{-h_{\tau}}, Q_{\tau} K^{2-3h_{\tau}} N_{\tau}^{-2h_{\tau}} 2^{2-2h_{\tau}} \right\} ( > 0 ). 
\end{align*}
Let $b \in I_{\tau}$ and $r > 0$ with $K \cdot \diam(\phi_{b}(X)) \leq r \leq r_{0}^{2}$ and we set $x := 1/b = \phi_{b}(0) \in \phi_{b}(X)$. 
To prove Theorem \ref{main2}, it suffices to show the following claim (see Theorem \ref{packingmeasurefiniteness}): 
\begin{claim*}[$\star$]
Let $m_{\tau}$ be the $h_{\tau}$-conformal measure of $S_{\tau}$. 
Then, 
\begin{align*} \label{sufficient_to_show}
    m_{\tau}(B(x, r)) \geq L_{\tau} r^{h_{\tau}}. 
\end{align*}
\end{claim*}
\noindent Rest of this paper, we consider the following three cases. 
\begin{enumerate}
    \item[] Case (1) $r \leq |x|/2$, 
    \item[] Case (2) $|x|/2 < r \leq 2|x|$ and 
    \item[] Case (3) $2|x| < r$. 
\end{enumerate}

We consider Case (1) $r \leq |x|/2$. 
Recall that by the assumption and Lemma \ref{basicestimate}, we have 
\begin{equation} \label{assumption_estimate}
    0 < r \ \leq \frac{|x|}{2} < |x| 
    \quad \text{and} \quad
    |x|^{2} = K \cdot K^{-1} |b|^{-2} \leq K \cdot \diam \phi_{b}(X) \leq r. 
\end{equation}
For simplicity, we set $f(z) := 1/z \ (z \in \mathbb{C} \setminus \{ 0 \})$ and  
\begin{align*}
    w := \frac{|x|^{2}}{|x|^{2} - r^{2}} \cdot \frac{1}{x} \quad \text{and} \quad  R := \frac{r}{|x|^{2} - r^{2}}. 
\end{align*}
Also, we set $I_{\tau}(x, r) := \{ a \in I_{\tau} \ | \ \phi_{a}(X) \subset B(x, r) \} = \{ a \in I_{\tau} \ | \ B(a + 1/2, 1/2) \subset f(B(x, r)) = B(w, R) \}$ (see Lemma \ref{inverse_circle})
and $I_{\tau, 1} := I_{\tau} (x, r) \cap B(0, |w|)$. 

Note that $\phi_{a}(X) \subset B(x, r)$ for each $a \in I_{\tau, 1}$. 
Besides, by the assumption of Case (1) and the definition of $I_{\tau, 1}$, 
we deduce that $|w| = |x|/(|x|^{2} - r^{2}) \leq 4/(3|x|)$  and for each $a \in I_{\tau, 1}$
\begin{align} \label{element_estimate}
    |a|^{-1} \geq |\omega|^{-1} \geq \frac{3}{4} \cdot |x|. 
\end{align}

Then, by the inequality (\ref{assumption_estimate}), we can show that $b \in I_{\tau, 1}$. 
Indeed, since $|b| = 1/|x| < |\omega|$, it is sufficient to show that $|w - (b + 1/2)| \leq R - 1/2$. 
By direct calculations and the inequality (\ref{assumption_estimate}), 
we have
\begin{align*}
    \left( R - \frac{1}{2} \right)^{2} - \left| w - \left( b + \frac{1}{2} \right) \right|^{2}
    = R^{2} - R - r^{2}|b|^{2} R^{2} + R r \Re(b)
    = R (r|b|^{2} - 1 + r \Re(b) )\geq 0, 
\end{align*}
%$|w - (b + 1/2)| \leq R - 1/2$ is equivalent to 
%\begin{align*}
%    R^{2} - R \geq R^{2} r^{2} |b|^{2} - R \cdot r \Re(b) 
%    \quad \Longleftrightarrow \quad 
%    r |b|^{2} - 1 + r \Re(b) \geq 0, 
%\end{align*}
where we use the equality $w - b = R r \cdot b$ and $(1- r^{2} |b|^{2})R = r |b|^{2}$, and $\Re(b) \ ( \geq 0 ) $ is the real part of $b$.  
Therefore, we have proved $b \in I_{\tau, 1}$ by the inequality (\ref{assumption_estimate}). 

%
%
%
%
%
%
%
%
%
%
%
%
%
%
%
%
%
%
%
%
%
%
%In addition, By the inequality (\ref{assumption_estimate}), 
We next show that 
\begin{align} \label{cardinality_estimate}
    \# (I_{\tau, 1}) > L^{\prime}_{\tau} \cdot r^{2} |x|^{-4}. 
\end{align}
To this end, 
%note that $R^{2} > r^{2} |x|^{-4}$. 
%If $R \geq C_{\tau}^{\prime} + 1$, 
if $R \geq C_{\tau}^{\prime} + 1$, 
then by the definition of $I_{\tau, 1}$, we have 
\begin{align*}
    I_{\tau, 1} 
    &= \{ a \in I_{\tau} \ | \ B(a + 1/2, 1/2) \subset B(w, R) \} \cap B(0, |w|) \notag \\
    &\supset \{ a \in I_{\tau} \ | \ a \in B(w, R-1) \} \cap B(0, |w|) 
    = I_{\tau} \cap B(w, R-1) \cap B(0, |w|). 
\end{align*}
Recall that $w \in E_{\tau} (\mathbb{R}_{+}^{2})$, $R -1 \geq C_{\tau}^{\prime}$ and $|w| = |x|/ (|x|^{2} - r^{2})> R > R -1$ by the inequality (\ref{assumption_estimate}). 
By Lemma \ref{latticepointcase1lemma} and the comment after the proof of Lemma \ref{latticepointcase1lemma}, we have 
\begin{align*}
    \# ( I_{\tau, 1} ) 
    &\geq \# ( I_{\tau} \cap B(w, R-1) \cap B(0, |w|) )
    > Q_{\tau}^{\prime} (R-1)^{2} \geq \frac{Q_{\tau}^{\prime}}{4} R^{2} 
    \geq L^{\prime}_{\tau} R^{2}
    > L^{\prime}_{\tau} r^{2} |x|^{-4}. 
\end{align*}
If $C_{\tau}^{\prime} + 1 > R$, then since $b \in I_{\tau, 1}$ (which is deduced by the inequality (\ref{assumption_estimate})) we have 
\begin{align*}
    \# (I_{\tau, 1}) 
    \geq 1 > \left( \frac{R}{C_{\tau}^{\prime} + 1} \right)^{2} 
    \geq L^{\prime}_{\tau} R^{2}
    > L^{\prime}_{\tau} r^{2} |x|^{-4}. 
\end{align*}
Therefore, we have proved the inequality (\ref{cardinality_estimate}). 
%by the inequality (\ref{assumption_estimate}). 

%
%
%
%
%
%
%
%
%
%
%
%
%
Now, by the inequalities (\ref{useful_inequality}) with $I := I_{\tau, 1}$, (\ref{element_estimate}),  (\ref{cardinality_estimate}), $|x|^{2} \leq r$ (in (\ref{assumption_estimate})), and $h_{\tau} - 2 < 0$ (see Lemma \ref{GCCFisCIFS}), it follows that 
\begin{align}
    m_{\tau}(B(x, r)) 
    &\geq m_{\tau}\left( \bigcup_{a \in I_{\tau, 1}} \phi_{a}(X) \right)
    = \sum_{a \in I_{\tau, 1}} K^{-h_{\tau}} |a|^{-2h_{\tau}}
    \geq \sum_{a \in I_{\tau, 1}} K^{-h_{\tau}} \left( \frac{9}{16} \right)^{h_{\tau}} |x|^{2h_{\tau}} \notag \\
    &\geq \# (I_{\tau, 1}) \cdot (2K)^{-h_{\tau}} |x|^{2h_{\tau}}
    \geq L^{\prime}_{\tau} (2K)^{-h_{\tau}} \cdot r^{2} \ |x|^{2h_{\tau}-4} 
    \geq L^{\prime}_{\tau} (2K)^{-h_{\tau}} \cdot r^{2} r^{h_{\tau}-2} \notag \\
    &\geq L^{\prime}_{\tau} (2K)^{-h_{\tau}} r^{h_{\tau}} 
    \geq L_{\tau} r^{h_{\tau}}. \label{final_inequality} 
\end{align} 
Thus, we have proved the statement of Claim ($\star$) for Case (1) $r \leq |x|/2$. 

We next consider Case (2) $|x|/2 < r \leq 2|x|$. 
We set $\tilde{r} := r/4$. 
Then, by the assumption, we have $\tilde{r} \leq |x|/2$. 
In addition, by the inequality $|x|^{2} = K \cdot K^{-1} |b|^{-2} \leq K \cdot \diam \phi_{b}(X) \leq r_{0}^{2}$ (see Lemma \ref{basicestimate} ), the definition of $r_{0}$ and the assumption, we have
\begin{align*}
    |x|^{2} \leq r_{0} \cdot |x| \leq \frac{1}{8} \cdot 2r = \tilde{r}.  
\end{align*}
Therefore, the positive real number $\tilde{r} > 0$ satisfies the inequalities $\tilde{r} \leq |x|/2$ (the assumption of Case (1)), 
$\tilde{r} < |x|$ and $|x|^{2} \leq \tilde{r}$ (the inequalities (\ref{assumption_estimate})) instead of $r >0$. 
By the same argument as the proof of $b \in I_{\tau}(x, r)$ and the inequalities (\ref{element_estimate}), (\ref{cardinality_estimate}) and (\ref{final_inequality}) in Case (1) with $\tilde{r} > 0$ instead of $r > 0$, we have
\begin{align*}
    m_{\tau}(B(x, r)) &\geq m_{\tau}(B(x, \tilde{r})) 
    \geq L^{\prime}_{\tau} (2K)^{-h_{\tau}} \tilde{r}^{h_{\tau}} 
    \geq L^{\prime}_{\tau} (2K)^{-h_{\tau}} 4^{-h_{\tau}} r^{h_{\tau}}
    \geq L^{\prime}_{\tau} (8K)^{-h_{\tau}} r^{h_{\tau}} \geq L_{\tau} r^{h_{\tau}}. 
\end{align*}
Thus, we have proved the statement of Claim ($\star$) for Case (2) $|x|/2 < r \leq 2|x|$. 

We finally consider Case (3) $2|x| < r$. 
We set $I_{\tau}(r) := \{ a \in I_{\tau} \ | \ r/N_{\tau} \leq 2 K |a|^{-1} < r \}$. 
Note that $ B(0, r/2) \subset B(x, r)$ by the assumption and the inequality $|y - x| \leq |x| + |y| < r/2 + r/2 = r$, where $y$ is an element of $B(0, r/2)$. 

Also, note that $\phi_{a}(X) \subset B(0, K|a|^{-1}) \subset B(0, r/2)$ for each $a \in I_{\tau}(r)$ by Lemma \ref{basicestimate}, and $|a|^{-1} \geq (2 K N_{\tau})^{-1} r$ for each $a \in I_{\tau}(r)$. 
In addition, we have 
\begin{align} \label{taulatticepointestimate}
    \# (I_{\tau}(r)) > 4 Q_{\tau} K^{2} r^{-2}. 
\end{align}
To show this, note that $2K r^{-1} > C_{\tau}$ since $r/2 < r_{0} \leq K C_{\tau}^{-1}$, where $ C_{\tau}$ is the number in Lemma \ref{lattice_point_in_annulus}. 
It follows that 
\begin{align*}
    \# (I_{\tau}(r) ) 
    &= \# ( \{ a \in I_{\tau} \ | \ 2K r^{-1} < |a| \leq 2 N_{\tau} K r^{-1} \} ) 
    = \# ( I_{\tau} \cap D_{2}(\tau, 2 K r^{-1}) )  
    > 4 Q_{\tau} K^{2} r^{-2}. 
\end{align*}
Therefore, we have proved the inequality (\ref{taulatticepointestimate}). 
Now, by the inequalities (\ref{useful_inequality}) and (\ref{taulatticepointestimate}), we deduce that
\begin{align*}
    m_{\tau}(B(x, r)) 
    &\geq m_{\tau}(B(0, r/2)) 
    \geq m_{\tau}\left( \bigcup_{a \in I_{\tau}(r)} \phi_{a}(X) \right) 
    \geq \sum_{a \in I_{\tau}(r)} K^{-h_{\tau}} |a|^{-2h_{\tau}} \\
    &\geq \# (I_{\tau}(r)) \cdot K^{-h_{\tau}} \left( \frac{r}{2KN_{\tau}} \right)^{2h_{\tau}}  
    > 2^{2-2h_{\tau}} Q_{\tau} K^{2-3h_{\tau}} N_{\tau}^{-2h_{\tau}} \cdot r^{2h_{\tau}-2} \\
    &\geq Q_{\tau} K^{2-3h_{\tau}} N_{\tau}^{-2h_{\tau}} 2^{2-2h_{\tau}} r^{h_{\tau}} \geq L_{\tau} r^{h_{\tau}},  
\end{align*}
where we use the inequality $r^{2h_{\tau}-2} > r^{h_{\tau}}$ since $r < r_{0} \leq 1/8 <1$ and $h_{\tau} < 2$ (see Lemma \ref{GCCFisCIFS}). 
Thus, 
we have proved 
the statement of Claim ($\star$) for Case (3) $2|x| < r$. 

Hence, by the three cases (Cases (1) $\sim$ (3)), 
we have proved Theorem \ref{main2}. 
\end{proof}

\section*{Acknowledgement} 
The authors would like to thank Mariusz Urba\'{n}ski for helpful comments on \cite{MU}. 
The first author is supported by JST CREST Grant Number JPMJCR1913 and the second author is partially supported by JSPS Grant-in-Aid for Scientific Research (B) Grant Number JP 19H01790.

\end{document}